\documentclass[a4paper, 10pt]{amsart}
\usepackage[margin=1.2in, marginparwidth=0.7cm]{geometry}
\usepackage{amssymb, amsthm, amsfonts}
\usepackage[centertags]{amsmath}
\usepackage{eucal}
\usepackage{mathrsfs}
\usepackage{marginnote}
\usepackage[all,cmtip]{xy}
\usepackage{graphicx, eepic}
\usepackage[shortlabels]{enumitem} 
\usepackage{xcolor}
\usepackage{bbm}
\usepackage{pdfpages}
\usepackage{titletoc}
\usepackage{booktabs}
\usepackage{mathtools}
\usepackage{thmtools}
\usepackage{multirow, url}
\usepackage{textcomp}
\DeclareMathAlphabet{\cmcal}{OMS}{cmsy}{m}{n}

\usepackage{lmodern}
\usepackage[utf8]{inputenc}
\usepackage[T1]{fontenc}

\definecolor{blue1}{rgb}{0, 0, 2}
\definecolor{sky}{rgb}{0, 0.2, 0.8}
\usepackage[colorlinks=true, citecolor=sky, linkcolor=magenta, urlcolor=blue1, filecolor=cyan, backref=page
]{hyperref}
\newtheorem{thm}{Theorem}[section]
\newtheorem{cor}[thm]{Corollary}
\newtheorem{lem}[thm]{Lemma}
\newtheorem{prop}[thm]{Proposition}

\theoremstyle{definition}
\newtheorem{defn}[thm]{Definition}

\theoremstyle{remark}
\newtheorem{rem}[thm]{\bf{Remark}}

\numberwithin{equation}{section} \numberwithin{table}{subsection}
	
\newtheorem*{thm*}{\bf{Theorem}}
\newtheorem*{claim*}{\bf{Claim}}
\newtheorem*{rem*}{\bf{Remark}}
\newtheorem*{rems*}{\bf{Remarks}}
\newtheorem*{exam*}{\bf{Example}}
\newtheorem*{exams*}{\bf{Examples}}

\newcommand{\F}{{\mathbb{F}}}

\newcommand{\Q}{{\mathbb{Q}}}

\newcommand{\T}{{\mathbb{T}}}

\newcommand{\Z}{{\mathbb{Z}}}

\newcommand{\m}{{\mathfrak{m}}}
\newcommand{\fn}{{\mathfrak{n}}}

\newcommand{\fs}{{\mathfrak{s}}}

\newcommand{\cC}{{\cmcal{C}}}
\newcommand{\cD}{{\cmcal{D}}}
\newcommand{\cE}{{\cmcal{E}}}

\newcommand{\cJ}{{\cmcal{J}}}

\newcommand{\cN}{{\cmcal{N}}}

\newcommand{\cT}{{\cmcal{T}}}

\newcommand{\cX}{{\cmcal{X}}}

\newcommand{\bfA}{\mathbf{A}}
\newcommand{\bfB}{\mathbf{B}}

\newcommand{\bfF}{\mathbf{F}}

\def\a{\alpha}
\def\b{\beta}

\def\d{\delta}

\def\g{\gamma}

\def\s{\sigma}

\def\S{\Sigma}

\def\<{\left\langle}
\def\>{\right\rangle}

\newcommand{\zmod}[1]{{\Z/{#1}\Z}}

\newcommand{\arinj}{\ar@{^(->}}
\newcommand{\arsurj}{\ar@{->>}}
\newcommand{\arsub}{\ar@{}[r]|-*[@]{\subset}}
\newcommand{\arsup}{\ar@{}[r]|-*[@]{\supset}}
\newcommand{\arcap}{\ar@{}[d]|-*[@]{\subset}}
\newcommand{\arcup}{\ar@{}[u]|-*[@]{\subset}}
\newcommand{\arin}{\ar@{}[u]|-*[@]{\in}}

\renewcommand{\pmod}[1]{{\,(\operatorname{mod}\hspace{0.7mm} {#1})}}

\newcommand{\Hom}{{\operatorname{Hom}}}
\newcommand{\Gal}{{\operatorname{Gal}}}
\newcommand{\Div}{{\operatorname{Div}}}

\newcommand{\End}{{\operatorname{End}}}
\newcommand{\Aut}{{\operatorname{Aut}}}

\newcommand{\sHom}{{\mathscr{H}\kern-.5pt om}}
\newcommand{\sExt}{{\mathscr{E}\kern-.5pt xt}}

\newcommand{\Frob}{{\operatorname{Frob}}}

\renewcommand{~}{\hspace*{0.3mm}}

\newcommand{\ts}{{\textsection}}

\newcommand{\ov}{\overline}

\mathchardef\hyp="2D
\newcommand{\xyv}[1]{\xymatrixrowsep{#1 pc}}
\newcommand{\xyh}[1]{\xymatrixcolsep{#1 pc}}

\newcommand{\qa}{{\quad \text{and} \quad}}

\newcommand{\mat}[4]{
 \left(  \begin{smallmatrix} #1 & #2 \\ #3 & #4 \end{smallmatrix} \right)}

\begin{document}                                                                          
\title[The action of the Hecke operators on the component groups]{The action of the Hecke operators on the component groups of modular Jacobian varieties}

\author{Taekyung Kim}
\address{IBS Center for Geometry and Physics,
Mathematical Science Building, Room 108,
Pohang University of Science and Technology,
77 Cheongam-ro, Nam-gu, Pohang, Gyeongbuk, Republic of Korea 37673}
\email{Taekyung.Kim.Maths@gmail.com}  

\author{~~Hwajong Yoo}
\address{IBS Center for Geometry and Physics,
Mathematical Science Building, Room 108,
Pohang University of Science and Technology,
77 Cheongam-ro, Nam-gu, Pohang, Gyeongbuk, Republic of Korea 37673}
\email{hwajong@gmail.com}                                                                    
\maketitle

\begin{abstract}
For a prime number $q\geq 5$ and a positive integer $N$ prime to $q$, Ribet proved the action of the Hecke algebra on the component group of the Jacobian variety of the modular curve of level $Nq$ at $q$ is ``Eisenstein'', which means the Hecke operator $T_\ell$ acts by $\ell+1$ when $\ell$ is a prime number not dividing the level. In this paper, we completely compute the action of the Hecke algebra on this component group by a careful study of supersingular points with extra automorphisms.
\end{abstract}

%\setcounter{tocdepth}{2}
%\tableofcontents

%%%%%%%%%%%%%%%%      Begining of the article    %%%%%%%%%%%%%%%%%%%%%

\section{Introduction}\label{sec: introduction}
Let $q\geq 5$ be a prime number, and let $N$ be a positive integer.
Let $X_0(Nq)$ denote the modular curve over $\Q$ and $J_0(Nq)$ its Jacobian variety. For any integer $n$, there is the Hecke operator $T_n$ acting on $J_0(Nq)$.
Let $\Phi_q(Nq)$ denote the component group of the special fiber $\cJ$ of the N\'eron model of $J_0(Nq)$ at $q$. 
According to the theorems of Ribet \cite{R88, R90} (when $q$ does not divide $N$) and Edixhoven \cite{Ed91} (in general), 
the action of the Hecke algebra on $\Phi_q(Nq)$ is ``Eisenstein.''
Here by ``Eisenstein'' we mean the Hecke operator $T_\ell$ acts on $\Phi_q(Nq)$ by $\ell+1$ when a prime number $\ell$ does not divide $Nq$.\footnote{On the other hand, Ribet and Edixhoven did not proceed to compute the action of the Hecke operator $T_p$ on $\Phi_q(Nq)$ for a prime divisor $p$ of the level $Nq$ because their results were enough for their applications.} In this article, we compute the action of the Hecke operators $T_\ell$ on the component group $\Phi_q(Nq)$ when $\ell$ divides $Nq$ and $q$ does not divide $N$. 

Here is an exotic example\footnote{this phenomenon cannot occur when the residual characteristic is greater than $3$} which leads us to this study: Let $N=\prod_{i=1}^\nu p_i$ be the product of distinct prime numbers with $\nu\geq 1$, and let $q\equiv 2 \text{ or } 5 \pmod 9$ be an odd prime number. Assume that $p_i \equiv 4 \text{ or } 7 \pmod 9$ for all $1 \leq i \leq \nu$. 
Let $\T(Nq)$ (resp. $\T(N)$) denote the $\Z$-subalgebra of $\End(J_0(Nq))$ (resp. $\End(J_0(N))$) generated by all the Hecke operators $T_n$ for $n\geq 1$. Let 
\[
\m:=(3, T_{p_i}-1, T_q+1, T_\ell-\ell-1 : \text{ for all } 1 \leq i \leq \nu, \text{ and for primes } \ell \nmid Nq \} \subset \T(Nq)
\]
and
\[
\fn:=(3, T_{p_i}-1, T_\ell-\ell-1 : \text{ for all } 1 \leq i \leq \nu, \text{ and for primes } \ell \nmid N \} \subset \T(N)
\]
be Eisenstein ideals. By \cite[Theorem 1.4]{Yoo3}, $\m$ is maximal. Furthermore, $\fn$ is maximal if and only if $\nu \geq 2$.
 
As observed by the second author \cite{Yoo10}, the dimension of $J_0(N)[\fn]=\nu$ if $\fn$ is maximal, i.e., $\nu \geq 2$. (Here $J_0(N)[\fn]:=\{ x \in J_0(N)(\ov{\Q}) : Tx =0 \text{ for all } T \in \fn \}$.)
It is an extension of $\mu_3^{\oplus {\nu-1}}$ by $\zmod 3$, 
and it does not contain a submodule isomorphic to $\mu_3$. On the other hand, the dimension of $J_0(Nq)[\m]$ is either $2\nu$ or $2\nu+1$. Furthermore $J_0(Nq)[\m]$ contains a submodule $\cN$ isomorphic to $J_0(N)[\fn]$, and it also contains $\mu_3^{\oplus \nu}$ (which is contributed from the Shimura subgroup). As $\cN$ is unramified at $q$, by Serre-Tate \cite{ST68}, $\cN$ maps injectively into $\cJ [\m]$ and it turns out that its image is isomorphic to $\cJ^0[\m]$, where $\cJ^0$ is the identity component of $\cJ$. 
(Note that $\Phi_q(Nq)$ is the quotient of $\cJ$ by $\cJ^0$.)
Since $\mu_3^{\oplus \nu}$ is also unramified at $q$, it maps into $\cJ[\m]$ and therefore its image maps injectively to $\Phi_q(Nq)[\m]$. (This statement is also true when $\nu=1$.) The structure of the component group $\Phi_q(Nq)$ is known by the work of Mazur and Rapoport \cite{MR77}\footnote{there are some minor errors in the paper, which are corrected by Edixhoven \cite[\ts 4.4.1]{Ed91}}: 
\[
\Phi_q(Nq)=\Phi \oplus (\zmod 3)^{2^\nu-1},
\]
where $\Phi$ is cyclic and generated by the image of the cuspidal divisor $(0)-(\infty)$. The action of the Hecke operators on $\Phi$ is well-known (e.g. \cite[Appendix A1]{Yoo2}), and so $\Phi[\m]=0$. Therefore $(\zmod 3)^{2^\nu-1}[\m] \neq 0$ and its dimension is at least $\nu$. Indeed it is equal to $2^{\nu-1}$, which can easily be computed by the theorems below.

Now, we introduce our results. The first one is as follows:
\begin{thm}\label{thm: main theorem Tp}
For a prime divisor $p$ of $N$, the Hecke operator $T_p$ acts on $\Phi_q(Nq)$ by $p$.
\end{thm}
The key idea of the proof is that the two degeneracy maps coincide on the component group (cf. \cite{R88},  \cite[Lemme 2 of \ts 4.2]{Ed91}).

Now, the missing one is the action of the Hecke operator $T_q$ on $\Phi_q(Nq)$. Note that $T_q$ acts on $\Phi_q(Nq)$ by an involution because the action of the Hecke algebra on $\Phi_q(Nq)$ is ``$q$-new.'' To describe its action more precisely, we need notation: 
For $N=\prod_{p \mid N} p^{n_p}$ being the prime factorization of $N$ (i.e., $n_p >0$), let $\nu:=\#\{ p : p \neq 2, 3 \}$ and let
\begin{align*}
 u &:=
  \begin{cases}
   0        & \text{if } q \equiv 1 \pmod 4 ~~\text{ or } ~~4 \mid N ~~\text{ or } ~~\exists ~p \equiv -1 \pmod 4 \\
   1        & \text{ otherwise},
  \end{cases}
  \\
 v &:=
  \begin{cases}
   0 & \text{if } q \equiv 1 \pmod 3 ~~\text{ or } ~~9 \mid N ~~\text{ or } ~~\exists ~p \equiv -1 \pmod 3\\
   1       & \text{otherwise}.
  \end{cases}
\end{align*}

Suppose that $(u, v)=(0, 0)$ or $\nu=0$. Then $\Phi_q(Nq)=\Phi$ and $T_q$ acts on $\Phi$ by $1$,
where $\Phi$ is the cyclic subgroup generated by the image of the cuspidal divisor $(0)-(\infty)$ (Proposition \ref{prop: case 1}).
If $\nu \geq 1$, $\Phi_q(Nq)$ becomes isomorphic to 
\[
\Phi' \oplus \bfA \oplus \bfB,
\]
where $\bfA \simeq (\zmod 2)^{\oplus u(2^\nu-2)}$, $\bfB \simeq (\zmod 3)^{\oplus v(2^\nu-1)}$ and $\Phi'$ is a cyclic group containing $\Phi$ and $\Phi'/\Phi \simeq (\zmod {2^u})$.\footnote{The structure of $\Phi_q(Nq)$ is already known by Mazur and Rapoport \cite{MR77} when $N$ is square-free and prime to $6$, and by Edixhoven \cite[\ts 4.4.1]{Ed91} in general.}
 
\begin{thm}\label{thm: main theorem Tq}
Assume that $(u, v)\neq (0, 0)$ and $\nu \geq 1$.
\begin{enumerate}
\item
Suppose that $v=1$.
Then there are distinct subgroups $B_i \simeq \zmod 3$ of $\bfB$ so that $\bfB = \oplus B_i$. For any $1\leq i \leq (2^\nu-1)$, $T_q$ acts on $B_i$ by $(-1)^i$. 

\item
Suppose that $u=1$.
Then there are distinct subgroups $A_i \simeq \zmod 2$ of $\bfA$ so that  $\bfA = \oplus A_i$. 
For any $1 \leq k \leq (2^{\nu-1}-2)$, $T_q$ acts on $A_{2k-1}\oplus A_{2k}$ by the matrix $\mat 1 0 1 1$.\footnote{This remind us the result by Mazur \cite{M77}: when $N$ is a prime number, the kernel of the Eisenstein prime of $J_0(N)$ containing a prime number $\ell$ is completely reducible when $\ell$ is odd, and is indecomposable when $\ell=2$.} In other words, if $A_{2k-1}=\<\mathbf{u}_{2k-1}\>$ and $A_{2k}=\<\mathbf{u}_{2k}\>$, then
\begin{equation*}
T_q(\mathbf{u}_{2k-1})=\mathbf{u}_{2k-1}+\mathbf{u}_{2k} \qa T_q(\mathbf{u}_{2k})=\mathbf{u}_{2k}.
\end{equation*} 
\end{enumerate}
\end{thm}
For a complete description of the action of $T_q$ on each subgroups, see Section \ref{sec: Tq action}.

\section{Supersingular points of $X_0(N)$} \label{sec: supersingular}
From now on, we always assume that \textsf{$q\geq 5$ is a prime number} and \textsf{$N$ is a positive integer which is prime to $q$}.
Let \textsf{$p$ denote a prime divisor of $N$}. Let $\bfF$ be an algebraically closed field of characteristic $q$. 

Let $\S(N)$ denote the set of supersingular points of $X_0(N)(\bfF)$.
Since we assume that $q\geq 5$, the group of automorphisms of supersingular points is cyclic of order $2$, $4$ or $6$. Let
\[
\S_n(N) := \{ s \in \S(N) : \# \Aut(s)=n \} \qa s_n(N) := \#\S_n(N).
\]
Note that $s_4(N)=u\cdot 2^\nu$ and $s_6(N)=v \cdot  2^\nu$ (cf. \cite[\ts 4.2, Lemme 1]{Ed91}), 
where $u, v$ and $\nu$ are as in Section \ref{sec: introduction}. 
Moreover $s_2(N)$ can be computed using Eichler's mass formula \cite[Theorem 12.4.5, Corollary 12.4.6]{KM85}:
\begin{equation} \label{eqn: mass formula}
\frac{s_2(N)}{2} + \frac{s_4(N)}{4}+\frac{s_6(N)}{6} = \frac{(q-1)Q}{24},
\end{equation}
where $Q:=N\prod_{p \mid N} (1+p^{-1})$ is the degree of the degeneracy map $X_0(N) \to X_0(1)$.

In the remainder of this section, we study $\S_4(N)$ and $\S_6(N)$ in detail. (See also \cite[\ts 2]{R88}, \cite[\ts 4]{R97} or \cite[\ts 4.2]{Ed91}.) In the section below, \textsf{we always assume that $\nu\geq 1$}, i.e., there is a prime divisor $p\geq 5$ of $N$. (If $\nu=0$ then $s_{2e}(N)\leq 1$ for $e=2$ or $3$, and the description is very simple.)

 Let $\cE$ be a supersingular elliptic curve with $\Aut(\cE)=\< \s \>$, and let $C$ be a cyclic subgroup of $\cE$ of order $N$. 
 Assume that $q \equiv -1 \pmod 4$ (resp. $q \equiv -1 \pmod 3$) if $\s=\s_4$ (resp. $\s=\s_6$), where $\s_k$ is a primitive $k$-th root of unity. 
\begin{prop}\label{prop:supersingular s6}
Let $N=p^n$ for some $n\geq 1$ with $p\geq 5$. Suppose $\Aut(\cE, C)=\< \s \>$. Then, there exists another cyclic subgroup $D$ of order $N$ such that $\cE[N]\simeq C \oplus D$. Moreover, $\Aut(\cE, D)=\< \s \>$ and $(\cE, C)$ is not isomorphic to $(\cE, D)$.
\end{prop}
\begin{proof}
Here, we closely follow the argument in the proof of Proposition 1 in \cite[\ts 2]{R88}.

Let $R$ be the subring $\Z[\s]$ of $\End(\cE, C)$. Since $\Aut(\cE, C)=\< \s \>$, $p\equiv 1 \pmod 4$ (resp. $p \equiv 1 \pmod 3$) if $\s=\s_4$ (resp. $\s=\s_6$). Therefore $p$ splits completely in $R$. 
Note that $R=\Z[\s]$ is a principal ideal domain and therefore
\[
R/{p R} \simeq R/{\g R} \oplus R/{\d R} \simeq {\d R}/{p R} \oplus {\g R}/{p R}
\]
with $p=\g \d$. Moreover,
\[
R/NR=R/{p^n R} \simeq R/{\g^n R} \oplus R/{\d^n R} \simeq {\d^n R}/{N R} \oplus {\g^n R}/{N R}.
\]

Note that $\cE[N]$ is a free of rank $1$ module over $R/NR$ by the action of $R$ on $\cE$.
We may identify $C$ with the quotient $I/NR$ for some ideal $I$ of $R$ containing $N$ if we fix an $R$-isomorphism between $\cE[N]$ and $R/NR$. Thus, $I=\d^n R$ or $\g^n R$. 
Suppose that $I=\d^n R$. Then, by the fixed isomorphism $C=\cE[\g^n]$. Let $D:=\cE[\d^n]$ so that its corresponding ideal is $\g^n R$. Then, $\cE[N] \simeq C \oplus D$. Moreover since $\g^n R$ is also an ideal of $R$, $D$ is also stable under the action of $\s$. In other words,
$\Aut(\cE, D)=\< \s \>$.

Since $\Aut(\cE)=\< \s \>$ and $\s(C)=C$, $(\cE, C)$ cannot be isomorphic to $(\cE, D)$.
\end{proof}

From now on, we use the same notation as in the proof of Proposition \ref{prop:supersingular s6}.
\begin{defn}
By the above formulas for every $n\geq 1$ and $p \equiv 1 \pmod 4$ (resp. $p\equiv 1 \pmod 3$), there are precisely two cyclic subgroups $C$, $D$ of $\cE$ of order $p^n$ such that $\Aut(\cE, C) =\Aut (\cE, D)= \< \s \>$ (and $\cE[p^n] \simeq C\oplus D$)  if $\s=\s_4$ (resp. if $\s=\s_6$). Thus, for each $n\geq 1$ we define $\cC_{p^n}$ and $\cD_{p^n}$ by 
\[
\cC_{p^n}:=\cE[\g^n] \qa \cD_{p^n}:=\cE[\d^n].
\]
\end{defn}

\begin{prop}\label{prop: degeneracy alpha}
For each $n\geq 1$, $\cC_{p^{n+1}}[p^{n}]=\cC_{p^{n}}$ and $\cD_{p^{n+1}}[p^{n}]=\cD_{p^{n}}$.
\end{prop}
\begin{proof}
By the fixed $R$-isomorphism $\iota$ between $\cE[p^{n+1}]$ and $R/{p^{n+1} R}$, we identify $\cC_{p^{n+1}}$ with $I/{p^{n+1} R}$, where $I=\d^{n+1} R$. As $I$ is an ideal of $R$, $\g I = p (\d^n R) \subset I$ and
$I/{\g I} \simeq R/{\g R} \simeq \zmod p$. Therefore 
\[
\xymatrix{
\cC_{p^{n+1}}[p^n] \ar[r]^-{\iota} & \left( I/{p^{n+1} R} \right) [p^n] = \g I/{p^{n+1} R} \ar[r]_-{\times 1/p}^-{\sim} & (\d^n R)/{p^n R},
}
\]
which corresponds to $\cC_{p^n}$. Similarly, we prove that $\cD_{p^{n+1}}[p^n]=\cD_{p^n}$, and the proposition follows.
\end{proof}

Let $N=Mp^n$ with $(6M, p)=1$ and $n \geq 1$. 
Let $L$ be a cyclic subgroup of $\cE$ of order $M$.
\begin{prop} \label{prop: degeneracy beta}
Suppose that $\Aut(\cE, \cC_{p^{n+1}}, L)=\< \s \>$. Then, there is an isomorphism between \\
$(\cE/{\cC_p},  \cC_{p^{n+1}}/{\cC_p}, (L\oplus \cC_p)/{\cC_p}))$ and $(\cE, \cC_{p^{n}}, L)$.
\end{prop}
\begin{proof}
We mostly follow the idea of the proof of Proposition 2 in \cite[\ts 2]{R88}.

The endomorphism $\g$ sends $\cE[\g^{n+1}]=\cC_{p^{n+1}}$ to $\cE[\g^n]=\cC_{p^n}$, and $L$ to itself (because $L \cap \cE[p]=0$).
Now we denote by $\ov{\g}$ the map $\cE/{\cC_{p}} \to \cE$ induced by $\g$. Note that $\ov{\g}$ is an isomorphism  because $\cC_{p}$ is $\cE[\g]$, the kernel of $\g$.
By the above consideration, this isomorphism $\ov{\g}$ sends $(\cC_{p^{n+1}}/{\cC_p}, (L\oplus \cC_{p})/{\cC_{p}})$ to $(\cC_{p}, L)$ because $\cC_{p^{n+1}}/{\cC_p}$ and $(L\oplus \cC_{p})/{\cC_{p}}$ are the images of $\cC_{p^{n+1}}$ and $L$ by the quotient map $\cE \to \cE/{\cC[p]}$, respectively. Therefore $\ov{\g}$ gives rise to the desired isomorphism between triples.
\end{proof}
\begin{cor}\label{cor: bijection alpha and alpha=beta}
The map $(\cE, C, L) \to (\cE, C[p^n], L)$ induces a bijection between $\S_{2e}(Np)$ and $\S_{2e}(N)$, where $\s=\s_{2e}$. Moreover if $(\cE, C, L) \in \S_{2e}(Np)$, we have
\begin{equation*}
(\cE, C[p^n], L) \simeq (\cE/{C[p]}, C/{C[p]}, (L\oplus C[p])/{C[p]}).
\end{equation*} 
\end{cor}
The corollary tells us that two degeneracy maps
$\a_p$ and $\b_p$ in Section \ref{sec: hecke action on ss sets}
coincide on $\S_{2e}(Np)$, which is a generalization of \cite[\ts 4.2, Lemme 2]{Ed91}.

\begin{prop}\label{prop: frobenius on S4 S6}
Suppose that $\Aut(\cE, \cC_{p^n}, L)=\< \s \>$. Then, $\Frob(\cE)=\cE$ and $\Frob(\cC_{p^n}) = \cD_{p^n}$, where $\Frob$ is the Frobenius morphism in characteristic $q$. Furthermore, $\Frob^2(\cE, \cC_{p^n}, L)=(\cE, \cC_{p^n}, L)$.
\end{prop}
\begin{proof}
Since $\cE$ is isomorphic to the reduction of the elliptic curve with $j$-invariant $1728$ (resp. $0$) if $\s=\s_4$ (resp. $\s=\s_6$), the Frobenius morphism is an endomorphism of $\cE$ (cf. \cite[Chapter V, Examples 4.4 and 4.5]{Si86}). Moreover, the Frobenius morphism and $\s$ generates $\End(\cE)$, which is a quaternion algebra.
(Note that the degree of the Frobenius morphism is $q$.)  Since $\End(\cE)$ is a quaternion algebra, we have
\begin{equation*}
\s \circ \Frob = \Frob \circ \bar{\s} = \Frob \circ \s^{-1},
\end{equation*}
where $\bar{\s}$ denotes the complex conjugation in $R=\Z[\s]$. Analogously, we have
\begin{equation*}
\g \circ \Frob = \Frob \circ \bar{\g} = \Frob \circ \d.
\end{equation*}

Since $\s(\Frob(\cC_{p^n}))=\Frob(\s^{-1}(\cC_{p^n}))=\Frob(\cC_{p^n})$, $\Frob(\cC_{p^n})$ is also stable under the action of $\s$. Moreover $\cC_{p^n}$ does not intersect with the kernel of $\Frob$.
Thus, $\Frob(\cC_{p^n})$ is either $\cC_{p^n}$ or $\cD_{p^n}$.
As an endomorphism of $\cE$, $\g$ sends $\cC_{p^n}$ (resp. $\cD_{p^n}$) to $\cC_{p^{n-1}}$ (resp. $\cD_{p^n}$). Similarly, $\d$ maps $\cC_{p^n}$ (resp. $\cD_{p^n}$) to $\cC_{p^n}$ (resp. $\cD_{p^{n-1}}$). Therefore if $\Frob (\cC_{p^n})=\cC_{p^n}$, then
\begin{equation*}
\g \circ \Frob (\cC_{p^n})=\g (\cC_{p^n})=\cC_{p^{n-1}} \qa \Frob \circ \d (\cC_{p^n})=\Frob (\cC_{p^n})=\cC_{p^n},
\end{equation*}
which is a contradiction. Thus, we get $\Frob (\cC_{p^n})=\cD_{p^n}$.

Since every supersingular point can be defined over $\F_{q^2}$, the quadratic extension of $\F_q$, 
$\Frob^2$ acts trivially on $\S(N)$ (cf. \cite[Remark 3.5.b]{R90}), which proves the last claim.
\end{proof}
\begin{rem}
By taking $H=(\zmod N)^*$ in Lemma 1 of \cite{R97}, we can obtain a similar result if we show that the Atkin-Lehner style involution in \cite[\ts 4]{R97} is equal to the Frobenius morphism.
\end{rem}

\section{The action of $T_p$ on the component group}\label{sec: hecke action on ss sets}
Before discussing the action of the Hecke operators on the component group, we study it on the group of divisors supported on supersingular points, which we denote by $\Div(\S(N))$.

Let $N=Mp^n$ with $(M, p)=1$ and $n\geq 1$, and assume that $(N, q)=1$.
Let $\a_p, ~~\b_p : X_0(Npq) \rightrightarrows X_0(Nq)$ denote two degeneracy maps of degree $p$, defined by
\[
\a_p(E, C, L):=(E, C[p^n], L) \qa \b_p(E, C, L):=(E/{C[p]}, C/{C[p]}, (L+C[p])/{C[p]}),
\]
where $C$ (resp. $L$) denotes a cyclic subgroup of order $p^{n+1}$ (resp. $Mq$) in an elliptic curve $E$
(cf. \cite[\ts 13]{MR91}).
Let $T_p$ and $\xi_p$ be two Hecke correspondences defined by
\[
\xyv{1.5}
\xyh{0.5}
\xymatrix{
& X_0(Npq) \ar[dl]_-{\a_p}  \ar[dr]^-{\b_p}& \\
X_0(Nq) \ar@<.5ex>@{-->}[rr]^-{\xi_p} && X_0(Nq). \ar@<.5ex>@{-->}[ll]^-{T_p}
}
\]
By pullback, the Hecke correspondence $T_p$ (resp. $\xi_p$)  induces the Hecke operator $T_p:=\b_{p,*} \circ \a_p^*$ (resp. $\xi_p:=\a_{p,*} \circ \b_p^*$) on $J_0(Nq)$. 

The same description of the Hecke operator $T_p$ on $\Div(\S(N))$ as above works. In other words, 
we have two degeneracy maps\footnote{every elliptic curve isogenous to a supersingular one is also supersingular} $\a_p, \b_p : \S(Np) \rightrightarrows \S(N)$ of degree $p$, defined by
\[
\a_p(E, C, L):=(E, C[p^n], L) \qa \b_p(E, C, L):=(E/{C[p]}, C/{C[p]}, (L+C[p])/{C[p]}),
\]
where $C$ (resp. $L$) denotes a cyclic subgroup of order $p^{n+1}$ (resp. $M$) in a supersingular elliptic curve $E$ over $\bfF$. These maps induce the maps on their divisor groups:
\begin{equation*}
\xymatrix{
\Div(\S(N)) \ar@<.5ex>[r]^-{\a_p^*} \ar@<-.5ex>[r]_{\b_p^*} & \Div(\S(Np)) 
\ar@<.5ex>[r]^-{\a_{p, *}} \ar@<-.5ex>[r]_{\b_{p, *}}& \Div(\S(N))
}
\end{equation*}
and the Hecke operator $T_p$ (resp. $\xi_p$) can be defined by $\b_{p, *} \circ \a_p^*$ (resp. $\a_{p, *} \circ \b_p^*$). (For the details when $n=0$, see \cite[\ts 3]{R90}, \cite[p. 18--22]{Ra91}, \cite[\ts 4.1]{Ed91} or \cite[\ts 7]{Em02}. By the same method, we get the above description without further difficulties.)

Now, let $\Phi_q(Nq)$ denote the component group of the special fiber $\cJ$ of the N\'eron model of $J_0(Nq)$ at $q$. To compute the action of $T_p$ on it, we closely follow the method of Ribet (cf. \cite{R88}, \cite[\ts 2, 3]{R90}, \cite[\ts 1]{Ed91}).
Since $N$ is not divisible by $q$, the identity component $\cJ^0$ of $\cJ$ is a semi-abelian variety by Deligne-Rapoport \cite{DR73} and Raynaud \cite{Ra70}. Moreover, $\cJ^0$ is an extension of $J_0(N)_{\bfF} \times J_0(N)_{\bfF}$ by $\cT$, the torus of $\cJ^0$. Let $\cX$ be the character group of the torus $\cT$. By Grothendieck, there is a (Hecke-equivariant) monodromy exact sequence \cite{Gro72} (see also \cite[\ts 2, 3]{R90}, \cite{Ra91}, or \cite[\ts 4]{Il15}),
\[
\xymatrix{
0 \ar[r] & \cX \ar[r]^-\iota & \Hom(\cX^t, \Z) \ar[r] & \Phi_q(Nq) \ar[r] & 0.
}
\]
Here $\cX^t$ denotes the character group corresponding to the dual abelian variety of $J_0(Nq)$, which is equal to $J_0(Nq)$. 
Namely, $\cX^t=\cX$ as sets, but the action of the Hecke operator $T_\ell$ on $\cX^t$ is equal to the action of its dual $\xi_\ell$ on $\cX$ (cf. \cite{R88}, \cite[\ts 3]{R90} and \cite[\ts 7]{Em02}). 
Note that $\cX$ is the group of degree $0$ elements in $\Z^{\S(N)}$. For $s, t \in \S(N)$, let $e(s):=\frac{\#\Aut(s)}{2}$ and 
\[
\phi_s(t):=\begin{cases} e(s) & \text{ if } s=t, \\
 ~~~~~~ 0 & \text{otherwise,} \end{cases} 
\]
and extends via linearity, i.e., $\phi_s(\sum a_i t_i)=\sum a_i \phi_s(t_i)$.
Then, $\iota(s-t)=\phi_s-\phi_t$. 
Note also that $\Hom(\Z^{\S(N)}, \Z$) is generated by $\psi_s:=1/e(s) \phi_s$, 
and $\Hom(\cX^t, \Z)$ is its quotient by the relation:
\begin{equation*}
\sum_{s \in \S(N)} \psi_s=\sum_{s \in \S(N)} \frac{1}{e(s)}\phi_s=0.
\end{equation*}
(This is the minimal relation to make $\sum a_w \psi_w$ vanish for all the divisors of the form $s-t$, which are the generators of $\cX$.) For more details, see \cite[\ts 2, 3]{R90} or \cite{Ra91}.

In conclusion, the component group $\Phi_q(Nq)$ is isomorphic to 
\[
\Hom(\Z^{\S(N)}, \Z)/R,
\]
where $R$ is the set of relations:
\begin{equation}\label{eqn: relations}
R = \{ e(s)\psi_s = e(t)\psi_t  \text{ for any } s, t \in \S(N), \sum_{t \in \S(N)} \psi_t=0 \}.
\end{equation}
Let $\Psi_s$ denote the image of $\psi_s$ by the natural projection $\Hom(\Z^{\S(N)}, \Z) \to \Phi_q(Nq)$.
The Hecke operator $T_p$ acts on $\Hom(\Z^{\S(N)}, \Z)$ via the action of $\xi_p$ on $\Div(\S(N))$, i.e., 
\[
T_p(\psi_s)(t):=\psi_s(\xi_p(t))=\psi_{s}(\a_{p, *} \circ \b_p^*(t)).
\]
For $s \in \S(N)$, we temporarily denote $\a_p^*(s)= \sum_{i=1}^p A^i(s)$ and $\b_p^*(s)=\sum_{i=1}^p B^i(s)$ (allowing repetition). We note that if $e(s)=1$ then there is no repetition, i.e., $A^i(s) \not\simeq A^j(s)$ and $B^i(s) \not\simeq B^j(s)$ if $i \neq j$. If $e(s)=e>1$, then after renumbering the index properly we have
\begin{equation*}
e(A^i(s))=1 \text{ for } 1 \leq i \leq p-1 \qa e(A^p(s))=e.
\end{equation*}
Moreover, we have
\begin{equation*}
A^{e(k-1)+1}(s) \simeq \cdots \simeq A^{ek}(s) \text{ for } 1 \leq k \leq \frac{p-1}{e}, \text{ and } A^i(s) \not\simeq A^j(s) \text{ if } \left [\frac{i-1}{e} \right ] \neq \left [\frac{j-1}{e}\right ],
\end{equation*}
where $[x]$ denotes the largest integer less than or equal to $x$.
This can be seen as follows: Let $\s=\s_{2e}$, and let $s$ represent a pair $(\cE, C)$, where $C$ is a cyclic subgroup of $E$ of order $N$. Since $e(s)=e$, $\s(C)=C$. Suppose that $s' \in \S(Np)$ with $\a_{p, *}(s')=s$.
Then $s'$ represents a pair $(\cE, D)$ with $D[N]=C$.
If $\s(D)=D$, then $\Aut ([(\cE, D)])=\< \s \>$ and $(\cE, D) \not\simeq (\cE, D')$ if $D \neq D'$. 
(Note that there is a unique such $D$.)
On the other hand, if $\s(D) \neq D$ then 
\begin{equation*}
(\cE, D) \simeq (\cE, \s(D)) \simeq \cdots \simeq (\cE, \s^{e-1}(D)) \simeq (\cE, \s^e(D))=(\cE, D)
\end{equation*}
and $\Aut ([(\cE, D)])= \{ \pm 1 \}$. Thus, we can rearrange $A^i(s)$ as above. (Note that this can only be possible when $p \equiv 1 \pmod {2e}$, which is indeed true because $e(s)=e$.)

Now, we claim that $\phi_s(\a_{p, *}(t))=\phi_t(\a_p^*(s))$.
Indeed, $\phi_s(\a_{p,*}(t))$ is nonzero if and only if $t \in \{A^1(s), \dots, A^p(s) \}$. So, it suffices to show this equality when $t \in \{A^1(s), \dots, A^p(s) \}$. If $e(s)=1$, then there is no repetition and the claim follows clearly (both are $1$). Now, let $e(s)=e>1$. If $e(t)=1$, then $t=A^i(s)$ for some $1\leq i \leq p-1$. Since the number of repetition of $t=A^i(s)$ in $\{A^1(s), \dots, A^p(s) \}$ is $e$, the above equality holds.
If $e(t)=e$, then $t=A^p(s)$ and $\phi_s(\a_{p,*}(t))=e=\phi_t(\a_p^*(s))$, as claimed. 
Analogously, we have
\begin{equation*}
\phi_t(\b_{p, *}(s))=\phi_s(\b_p^*(t)).
\end{equation*}
More generally, we get
\begin{align*}
\phi_s(\a_{p,*}\circ \b_p^*(t))&=\sum_{i=1}^p \phi_s(\a_{p,*}(B^i(t)))
=\sum_{i=1}^p \sum_{j=1}^p \phi_{B^i(t)} (A^j(s))
=\sum_{j=1}^p \sum_{i=1}^p \phi_{A^j(s)} (B^i(t)) \\
&=\sum_{j=1}^p \phi_{A^j(s)} (\b_p^*(t))=\sum_{j=1}^p \phi_t(\b_{p,*}(A^j(s)))=\phi_t(\b_{p, *} \circ \a_p^*(s))=\phi_t(T_p(s)).
\end{align*}
If we set $T_p(s)=\sum s_i$, then $\phi_t(T_p(s))=\sum \phi_{s_i}(t)=\sum e(s_i) \psi_{s_i}(t)$ and hence for any $t \in \S(N)$,
\begin{equation*}
e(s)T_p(\psi_s)(t)=\phi_s(\a_{p,*}\circ \b_p^*(t))=\phi_t(T_p(s))=e(s_i)\psi_{s_i}(t).
\end{equation*}
In other words, we get
\begin{equation}
T_p(\Psi_s)=\frac{1}{e(s)} \sum e(s_i) \Psi_{s_i}.
\end{equation}

We can also define the action of $T_p$ on the component group via functorialities. Namely, let
\begin{equation*}
\xymatrix{
\Phi_q(Nq) \ar@<.5ex>[r]^-{\a_p^*} \ar@<-.5ex>[r]_{\b_p^*} & \Phi_q(Npq) 
\ar@<.5ex>[r]^-{\a_{p, *}} \ar@<-.5ex>[r]_{\b_{p, *}}& \Phi_q(Nq).
}
\end{equation*}
denote the maps functorially induced from the degeneracy maps\footnote{if $\a_p^*(s)=\sum t_j$ then $\a_p^*(\Psi_s)=\sum \Psi_{t_j}$ and if $\a_p(t)=s$ then $\a_{p, *}(\Psi_t)=e(s)/e(t)\Psi_s$; and similarly for $\b_p^*$ and $\b_{p,*}$}. Then, as before $T_p:=\b_{p, *} \circ \a_p^*$. Note that since the degrees of $\a_p$ and $\b_p$ are $p$, we have 
$\a_{p,*} \circ \a_p^* =\b_{p,*} \circ \b_p^* =p$.

\begin{lem}
$\a_{p, *} = \b_{p, *}$ on $\Phi_q(Npq)$.
\end{lem}
\begin{proof}
For $s \in \S_{2e}(Npq)$ with $e=2$ or $3$, $\a_p(s)=\b_p(s)$ by Remark \ref{cor: bijection alpha and alpha=beta}, and hence $\a_{p, *}(\Psi_s)=\b_{p, *}(\Psi_s)$. For $s \in \S_2(Npq)$, let $\a_p(s)=t$ and $\b_p(s)=w$. Then, $\a_{p, *}(\Psi_s)=e(t)\Psi_t = e(w)\Psi_w=\b_{p, *}(\Psi_s)$. In other words, for any $s \in \S(Npq)$, $\a_{p, *}(\Psi_s)=\b_{p, *}(\Psi_s)$. Since $\Psi_s$'s generates $\Phi_q(Npq)$, the result follows.
\end{proof}

In fact, Theorem \ref{thm: main theorem Tp} is an easy corollary of the above lemma.
\begin{proof}[Proof of Theorem \ref{thm: main theorem Tp}]
Since $\a_{p, *}=\b_{p, *}$ on $\Phi_q(Npq)$, we have
\begin{equation*}
T_p(\Psi_s)=\b_{p, *}\circ \a_p^*(\Psi_s)=\a_{p, *}\circ \a_p^*(\Psi_s)=p\Psi_s,
\end{equation*}
which implies the result.
\end{proof}

\section{The action of $T_q$ on the component group} \label{sec: Tq action}
In this section, we provide a complete description of the action of $T_q$ on the component group $\Phi_q(Nq)$. See Propositions \ref{prop: case 2}, \ref{prop: case 3} and \ref{prop: case 4}, which imply Theorem \ref{thm: main theorem Tq}.

Note that the Hecke operator $T_q$ acts on $\S(N)$ by the Frobenius morphism \cite[Proposition 3.8]{R90}, and the same is true for $\xi_q$. Since the Frobenius morphism is an involution on $\S(N)$ (cf. Proposition \ref{prop: frobenius on S4 S6}), we have
\begin{equation}
T_q(\psi_s)(t)=\psi_s(\xi_q(t))=\psi_s(\Frob (t))=\psi_{\Frob (s)}(t) \text{ for any } t \in \S(N),
\end{equation} 
which implies that $T_q(\psi_s)=\psi_{\Frob (s)}$.

From now on, if there is no confusion we remove $(N)$ from the notation for simplicity.
Let $n:=\frac{(q-1)Q}{12}$ (which is not necessarily an integer), and let $\Phi$ denote the cyclic subgroup of $\Phi_q(Nq)$ generated by $\Psi_\fs$ for a fixed $\fs \in \S_2$. (Note that this $\Phi$ is the same as that of Mazur and Rapoport \cite{MR77}, namely, $\Phi$ is equal to the cyclic subgroup generated by the image of the cuspidal divisor $(0)-(\infty)$.)

\subsection{Case 1: $(u, v)=(0, 0)$ or $\nu=0$} $~$

Let $e=1$ if $(u, v)=(0, 0)$ and $e=2u+3v$ if $(u, v)\neq (0, 0)$ and $\nu=0$. 
If $(u, v)=(0,0)$, $s_2=n$ and $s_4=s_6=0$. If $(u, v) \neq (0, 0)$ and $\nu=0$, then $s_{2e}=1$ and $s_2=\frac{en-1}{e}$. (Note that $s_2$ is an integer but $n$ is not.)
\begin{prop}\label{prop: case 1}
The component group $\Phi_q(Nq)$ is equal to $\Phi$, which is cyclic of order $en$. The Hecke operator $T_q$ acts on it by $1$. 
\end{prop}
\begin{proof}
First, we assume that $(u, v)=(0, 0)$. Then for any $s \in \S=\S_2$, $\Psi_s=\Psi_\fs$. Therefore $\Phi_q(Nq)=\Phi$ and $n\Psi_\fs=\sum_{s \in \S} \Psi_s=0$. 
Moreover, $T_q(\Psi_\fs)=\Psi_{s'}=\Psi_\fs$, where $s'=\Frob(\fs)$.

Now, we assume that $(u, v) \neq (0, 0)$ and $\nu=0$. In this case, either $N=2q$ (with $(u, v)=(1, 0)$ and $e=2$) or $N=3q$ (with $(u, v)=(0, 1)$ and $e=3$). In each case, let $z \in \S_{2e}$. Then
\begin{equation*}
\sum_{s \in \S_2}\Psi_s + \Psi_z=s_2 \Psi_\fs+\Psi_z=0 \qa \Psi_\fs=e \Psi_z.
\end{equation*}
Therefore the component group is generated by $\Psi_z$, and its order is $(es_2+1)=en$. Since $en=es_2+1$ is prime to $e$, this group is also generated by $\Psi_\fs=e\Psi_z$. (In fact, $\Psi_z=-s_2 \Psi_\fs$.) Moreover we have $T_q(\Psi_\fs)=\Psi_\fs$ as above.
\end{proof}

\subsection{Case 2: $(u, v)=(0, 1)$ and $\nu\geq 1$}$~$

In this case, $s_4=0$, $s_6=2^\nu$, and $s_2=\frac{3n-2^\nu}{3}$. Let $\S_6:=\{ t_1, t_2, \dots, t_{2^\nu} \}$. Here we assume that $\Frob(t_{2k-1})=t_{2k}$ for $1 \leq k \leq 2^{\nu-1}$.\footnote{By Proposition \ref{prop: frobenius on S4 S6}, we know that $\Frob$ is an involution of $\S_6$ without fixed points.} Let $t:=t_{2^\nu-1}$ and $t':=t_{2^\nu}$.

\begin{prop}\label{prop: case 2}
The component group $\Phi_q(Nq)$ decomposes as follows:
\[
\Phi_q(Nq)=\bigoplus_{i=0}^{2^\nu-1} B_i =: B_0 \oplus \bfB,
\]
where $B_0=\Phi$ is cyclic of order $3n$, and for $1\leq i \leq 2^\nu-1$, $B_i$ is cyclic of order $3$.
For $1\leq k \leq 2^{\nu-1}$, $B_{2k-1}$ and $B_{2k}$ are generated by 
\[
\mathbf{v}_{2k-1}:=\Psi_{t_{2k-1}}-\Psi_{t_{2k}} ~~\text{ and }~~\mathbf{v}_{2k}:=\Psi_{t_{2k-1}}+\Psi_{t_{2k}}-\Psi_t-\Psi_{t'}, \text{respectively}.
\]
The Hecke operator $T_q$ acts on $B_i$ by $(-1)^i$.
\end{prop}
\begin{proof}
Note that $\Psi_s=3\Psi_{t_i}=3\Psi_{t_j}$ for all $i, j$ and $\sum_{i=1}^{2^\nu} \Psi_{t_i} + s_2 \Psi_s=0$. Therefore $\Phi_q(Nq)$ is generated by $\Psi_{t_i}$ for $1\leq i \leq 2^\nu-1$. The order of each group $\<\Psi_{t_i}\>$ is $9n$ because
\[
9n \Psi_{t_i}= 3s_2(3\Psi_{t_i})+\sum_{i=1}^{2^\nu} 3\Psi_{t_i}=3\left(\sum_{s \in \S_2} \Psi_s +\sum_{i=1}^{2^\nu} \Psi_{t_i} \right)=0,
\]
and $9n$ is the smallest positive integer to make this happen.
Moreover $\< \Psi_{t_i} \> \cap \< \Psi_{t_j} \>$ is of order $3n$ for any $i\neq j$. 
Since $3n=3s_2+2^\nu$ is prime to $3$, we can decompose the component group into
\begin{equation}\label{equation}
\< 3\Psi_t \> \oplus \<(3s_2+2^\nu) \Psi_t \> \bigoplus_{i=1}^{2^\nu-2} \< \Psi_{t_i}-\Psi_t \>. 
\end{equation}
Since $\Psi_s=3\Psi_{t_i}=3\Psi_t=3\Psi_{t'}$ for any $i$ and $\sum_{i=1}^{2^\nu} \Psi_{t_i}=-3s_2\Psi_t$, we have
\[
\Psi_{2k-1}-\Psi_t=2\mathbf{v}_{2k-1}+2\mathbf{v}_{2k}+\mathbf{v}_{2^{\nu}-1};
\]
\[
\Psi_{2k}-\Psi_t=\mathbf{v}_{2k-1}+2\mathbf{v}_{2k}+\mathbf{v}_{2^{\nu}-1};
\]
\[
(3s_2+2^\nu)\Psi_t=\sum_{i=1}^{2^\nu} (\Psi_{t}-\Psi_{t_i})
=-\sum_{k=1}^{2^{\nu-1}} \mathbf{v}_{2k}-(-1)^\nu \mathbf{v}_{2^\nu-1}.
\]
Therefore the decomposition in the proposition is isomorphic to (\ref{equation}). The action of $T_q$ on each $B_i$ is obvious from its construction.
\end{proof}

\subsection{Case 3: $(u, v)=(1, 0)$ and $\nu\geq 1$}$~$

Note that $s_4=2^\nu$, $s_6=0$, and $s_2=n-2^{\nu-1}$.
Let $\S_4=\{w_1, w_2, \dots, w_{2^\nu} \}$. As before, we assume that $\Frob(w_{2k-1})=w_{2k}$ for $1 \leq k \leq 2^{\nu-1}$.\footnote{By Proposition \ref{prop: frobenius on S4 S6}, we know that $\Frob$ is an involution of $\S_4$ without fixed points.} Let $w:=w_{2^\nu-1}$ and $w':=w_{2^\nu}$. 

\begin{prop} \label{prop: case 3}
The component group $\Phi_q(Nq)$ decomposes as follows:
\[
\Phi_q(Nq)=\bigoplus_{i=0}^{2^\nu-2} A_i=A_0 \oplus \bfA,
\]
where $A_0$ is cyclic of order $4n$ generated by $\Psi_w$, and for $1\leq i \leq 2^\nu-2$, $A_i$ is cyclic of order $2$. 
For $1\leq k \leq 2^{\nu-1}-2$, $A_{2k-1}$ and $A_{2k}$ are generated by 
\[
\mathbf{u}_{2k-1}:=\Psi_{w_{2k-1}}-\Psi_w ~~\text{ and }~~\mathbf{u}_{2k}:=\Psi_{w_{2k-1}}+\Psi_{w_{2k}}-\Psi_w-\Psi_{w'}, ~\text{respectively}.
\]
And $A_{2^\nu-3}$ and $A_{2^\nu-2}$ are generated by 
\[
\mathbf{u}_{2^\nu-3} :=\Psi_{w_{2^\nu-3}}-\Psi_w ~~\text{ and }~~\mathbf{u}_{2^\nu-2}:=\Psi_{w_{2^\nu-3}}-\Psi_{w_{2^\nu-2}}, ~\text{respectively}.
\]
Moreover, the action of the Hecke operator $T_q$ on each group as follows:
\[
T_q(\Psi_w)=(1+2n)\Psi_w+\sum_{i=1}^{2^{\nu-1}-1} \mathbf{u}_{2i};
\]
\[
T_q(\mathbf{u}_{2k-1})=\mathbf{u}_{2k-1}+\mathbf{u}_{2k} \qa T_q(\mathbf{u}_{2k})=\mathbf{u}_{2k} ~~\text{ for }~ 1 \leq k \leq 2^{\nu-1}-2;
\]
\[
T_q(\mathbf{u}_{2^\nu-3})=2n \Psi_w+\mathbf{u}_{2^\nu-3}+\sum_{i=1}^{2^{\nu-1}-2} \mathbf{u}_{2i} \qa T_q(\mathbf{u}_{2^\nu-2})=\mathbf{u}_{2^\nu-2}.
\]
\end{prop}
\begin{proof}
  The argument in Proposition \ref{prop: case 2} applies \textit{mutatis mutandis}. For instance, when $\nu\geq 2$ an isomorphism between $A_0 \bigoplus_{i=1}^{2^\nu-2} \<\Psi_{w_i}-\Psi_w\>$ and $A_0\oplus \bfA$ can be given by the following data: 
for $1 \leq k \leq 2^{\nu-1}-2$, 
\[
\Psi_{w_{2k}}-\Psi_w=\mathbf{u}_{2k}+\mathbf{u}_{2k-1}+(\Psi_{w'}-\Psi_w) \text{ and } 
\Psi_{w}-\Psi_{w'}=2n \Psi_w+\sum_{i=1}^{2^{\nu-1}-1} \mathbf{u}_{2i};
\]
\[
\Psi_{w_{2^\nu-2}}-\Psi_w=\mathbf{u}_{2^\nu-3}+\mathbf{u}_{2^\nu-2}.
\]
The action of Hecke operator $T_q$ on each $A_i$ is clear except 
\[
T_q(\Psi_w)=\Psi_{w'}=\Psi_{w}-(\Psi_{w}-\Psi_{w'})=(1+2n)\Psi_w+\sum_{i=1}^{2^{\nu-1}-1} \mathbf{u}_{2i},
\]
\[
T_q(\mathbf{u}_{2^\nu-3})=\Psi_{w_{2^\nu-2}}-\Psi_{w'}=\mathbf{u}_{2^\nu-3}+\mathbf{u}_{2^\nu-2}+(\Psi_{w}-\Psi_{w'})=2n\Psi_{w}+\mathbf{u}_{2^\nu-3} + \sum_{i=1}^{2^{\nu-1}-2} \mathbf{u}_{2i}.
\]
\end{proof}

\subsection{Case 4: $(u, v)=(1, 1)$ and $\nu\geq 1$}$~$

Note that $s_4=s_6=2^\nu$ and $s_2=\frac{6n-5\cdot 2^\nu}{6}$.
Let $\S_4=\{w_1, \dots, w_{2^\nu} \}$ and $\S_6:=\{ t_1,  \dots, t_{2^\nu} \}$. As before, we assume that $\Frob(w_{2k-1})=w_{2k}$ and $\Frob(t_{2k-1})=t_{2k}$ for $1 \leq k \leq 2^{\nu-1}$. Let $w:=w_{2^\nu-1}$ and $w':=w_{2^\nu}$. Also, let $t:=t_{2^\nu-1}$ and $t':=t_{2^\nu}$.

\begin{prop}\label{prop: case 4}
The component group $\Phi_q(Nq)$ decomposes as follows:
\[
\Phi_q(Nq)=A_0 \oplus \bfA \oplus \bfB,
\]
where $A_0$ is cyclic of order $12n$ generated by $\Psi_w$.
The structures of $\bfA$ and $\bfB$ are the same as those in Propositions \ref{prop: case 2} and \ref{prop: case 3}. The actions of $T_q$ on $\bfA$ and $\bfB$ are the same as before except on $A_{2^\nu-3}$ (when $\nu \geq 2$), where $T_q$ acts by
\[
T_q(\mathbf{u}_{2^\nu-3})=6n \Psi_w+\mathbf{u}_{2^\nu-3}+\sum_{i=1}^{2^{\nu-1}-2} \mathbf{u}_{2i}.
\]
Moreover, the action of $T_q$ on $A_0$ is analogous as before:
\[
T_q(\Psi_w)=(1+6n)\Psi_w+\sum_{i=1}^{2^{\nu-1}-1} \mathbf{u}_{2i}.
\]
\end{prop}
\begin{proof}
Note that from (\ref{eqn: relations}) we have
\[
s_2\Psi_s+\Psi_{w_1}+\cdots+\Psi_{w'}+\Psi_{t_1}+\cdots+\Psi_{t'}=0.
\]
Multiplying by $3$, we have
\begin{equation}\label{eqn: 2}
\Psi_{w_1}+\cdots+\Psi_{w'}=-(3s_2+2\cdot 2^{\nu})\Psi_s=-(6s_2+4 \cdot 2^\nu)\Psi_w.
\end{equation}
Also, multiplying by $4$, we have
\begin{equation}\label{eqn: 3}
\Psi_{t_1}+\cdots+\Psi_{t'}=-(4s_2+3\cdot 2^\nu)\Psi_s=-(12s_2+9\cdot 2^{\nu})\Psi_t.
\end{equation}
Therefore $\Psi_{w_1}, \dots, \Psi_w, \Psi_{t_1}, \dots, \Psi_t$ can generate the whole group.
By the similar computation, the order of $\< \Psi_{w_i} \>$ is $12n$ and the order of $\<\Psi_{t_i} \>$ is $18n$. All of them contain $\Phi$ as a subgroup, which is of order $6n$. Here we note that $\< \Psi_t \>=\< 3\Psi_t \> \oplus \< 6n\Psi_{t} \>$ because $6n=6s_2+5\cdot 2^\nu$ is prime to $3$.  
Therefore we can decompose $\Phi_q(Nq)$ into
\begin{equation}\label{eqn : 5}
\< \Psi_w \> \bigoplus_{i=1}^{2^\nu-2} \< \Psi_{w_i}-\Psi_w \> \bigoplus_{i=1}^{2^\nu-2} \< \Psi_{t_i}-\Psi_t \>\bigoplus \< 6n\Psi_t \>.
\end{equation}
As in Propositions \ref{prop: case 2} and \ref{prop: case 3}, we can find an isomorphism between (\ref{eqn : 5}) and $A_0\oplus\bfA \oplus \bfB$, which proves the first part. 
From (\ref{eqn: 2}) (and the previous discussions) we have
\[
\Psi_{w}-\Psi_{w'}= (6s_2+5\cdot 2^\nu)\Psi_w+\sum_{i=1}^{2^{\nu-1}-1} \mathbf{u}_{2i}=6n\Psi_w+\sum_{i=1}^{2^{\nu-1}-1} \mathbf{u}_{2i}.
\]
The action of $T_q$ on each components is also obvious except
\[
T_q(\Psi_w)=\Psi_{w'}=\Psi_w-(\Psi_{w}-\Psi_{w'})=(1+6n)\Psi_w+\sum_{i=1}^{2^{\nu-1}-1} \mathbf{u}_{2i} ~~\text{ and}
\]
\[
T_q(\mathbf{u}_{2^\nu-3})=\Psi_{w_{2^\nu-2}}-\Psi_{w'}=\mathbf{u}_{2^\nu-3}+\mathbf{u}_{2^\nu-2}+(\Psi_{w}-\Psi_{w'})=6n\Psi_{w}+\mathbf{u}_{2^\nu-3} + \sum_{i=1}^{2^{\nu-1}-2} \mathbf{u}_{2i}.
\]
\end{proof}

\subsection*{Acknowledgements}
The second author would like to thank Kenneth Ribet for a number of very helpful discussions about Eisenstein ideals and component groups. The anonymous referee deserves special thanks for a thorough reading of the manuscript and for many useful comments and suggestions.
This work was supported by IBS-R003-D1.

\bibliographystyle{annotation}

\end{document}